\documentclass[leqno,11pt,oneside]{amsart}
\usepackage{amsmath,amssymb,amsthm,mathrsfs}
\usepackage{epsfig,color}
\usepackage[latin1]{inputenc}
\usepackage[english]{babel}
\usepackage{mathtools}
\usepackage{braket}
\usepackage[toc]{appendix}
\usepackage{enumitem}
\usepackage{ulem}
\usepackage{graphicx}
\usepackage{xfrac, nicefrac}
\usepackage{csquotes}
\usepackage{comment}
\allowdisplaybreaks

\numberwithin{equation}{section}

\usepackage{esint}

\usepackage{hyperref}

\DeclarePairedDelimiter{\norm}{\lVert}{\rVert}

\newtheorem*{theorem*}{Theorem}
\newtheorem{theorem}{Theorem}[section]

\newtheorem*{proposition*}{Proposition}

\newtheorem{corollary}[theorem]{Corollary}
\newtheorem{remark}[theorem]{Remark}
\newtheorem*{remark*}{Remark}

\usepackage{graphicx}
\usepackage{tikz}
\usepackage{pgfplots}
\usepackage{geometry}

\newcommand{\field}[1]{\mathbb{#1}}

\providecommand{\R}{\field{R}}

\renewcommand{\emptyset}{\varnothing}

\begin{document}
	\author{Francesco Pagliarin and Nikita Simonov}
    
    \address{Francesco Pagliarin: Institut Camille Jordan, Universit\'e Claude Bernard Lyon 1, 43 boulevard du 11 Novembre 1918
69622 Villeurbanne cedex, France}
\email{pagliarin@math.univ-lyon1.fr}

\address{Nikita Simonov: Sorbonne Universit\'e, Universit\'e Paris Cit\'e, CNRS, Laboratoire Jacques-Louis Lions, LJLL, F-75005 Paris, France}
\email{nikita.simonov@sorbonne-universite.fr}
	\title{From weighted Carlson-Levin inequalities to weighted Gagliardo-Nirenberg inequalities}
\begin{abstract} 
We study weighted Carlson-Levin inequalities on cones, computing the optimal constant and classifying optimizers via an Entropy method. The limiting case of a logarithmic Carlson-Levin inequality is also treated. As a Corollary, we obtain a rigidity result for a weighted Gagliardo-Nirenberg inequality. We conclude with some examples involving monomial weights.
\end{abstract}

\maketitle
%\tableofcontents

\section{Introduction}
Sobolev and Gagliardo-Nirenberg inequalities, either in the classical form (i.e., posed on $\mathbb{R}^d$ with integrals taken with respect to the Lebesgues measure) or in their weighted counterpart (i.e., posed on a set $E\subset\mathbb{R}^d$ with integral taken with respect to the measure $\omega\,dz$ where $\omega$ is a non-negative weight), play a prominent role in the investigation of existence, regularity, and qualitative behaviour of many PDEs, for instance~\cite{EvansPDEBook,GisutiBook,Maggi_2023,Bonforte_2025,Bakry_2014} and references therein. For such inequalities the optimal constant and optimizers are fully understood only in some special instances,~\cite{Dolbeault_2003,Cordero_Erausquin_2004}. In the weighted case, essentially all the known results hold under geometric assumption on the weights (for instance, concavity, convexity, or homogeneity), see for instance~\cite{Cabr__2016,Cabr__2013}. In this direction,  Z. Balogh, S. Don and A. Krist\'aly considered, in a recent paper \cite{Balogh_2025}, a family of weighted Gagliardo-Nirenberg and Sobolev inequalities posed on an open convex cone $E\subseteq \mathbb{R}^d$ for functions $u: \,E\rightarrow \mathbb{R}$,
\begin{equation}\label{eqGN}\tag{$GNS_w$}
\mathcal{C}\,\left(\int_E |u|^{p\sigma}\omega_1dz\right)^\frac{1}{q}\leq \left(\int_E |\nabla u|^p\omega_2 dz\right)^\frac{1-\theta}{p}\left(\int_E |u|^{p\sigma\gamma}\omega_3 dz\right)^\frac{1}{s},
\end{equation}
where $\mathcal{C}>0$. The weights $\omega_i$, $i\in\{1,2,3\}$ are non-negative, homogeneous, and satisfy a compatibility condition (see \cite[Equation 1.6]{Balogh_2025} and \eqref{eqmain}). The parameters are as follows: $0<\gamma<1$, $\sigma=(p(\gamma-1)+1)^{-1}$, and the exponent $\theta$ is deduced by scaling. In~\cite{Balogh_2025}, inequality~\eqref{eqGN} is investigated via optimal transport theory by adapting an argument largely inspired by the celebrated work of D. Corder-Erasquin, B. Nazaret and C. Villani~\cite{Cordero_Erausquin_2004}. Let us define
\begin{equation*}
\mathcal{C}_{WGNS}:=\inf_{f\in\mathrm{W}^{1,p}_w(E)} \frac{\left(\int_E |\nabla u|^p\omega_2 dz\right)^\frac{1-\theta}{p}\left(\int_E |u|^{p\sigma\gamma}\omega_3 dz\right)^\frac{1}{s}}{\left(\int_E |u|^{p\sigma}\omega_1dz\right)^\frac{1}{q}}\,,
\end{equation*}
where $\mathrm{W}_w^{1,p}(E)$ is a weighted version of the standard space $\mathrm{W}^{1,p}(\mathbb{R}^d)$, see~\cite[(1.7)]{Balogh_2025} for an explicit definition. In~\cite[Theorem 1.1]{Balogh_2025} the authors prove constructively the validity of \eqref{eqGN} for an explicit constant $\mathcal{C}_{BDK}>0$ (depending on the weights and on the compatibility condition~\eqref{eqmain}) such that
\begin{equation*}
0 < \mathcal{C}_{BDK}\leq \mathcal{C}_{WGNS}\,.
\end{equation*}
In particular, when the three weights are equal (up to a multiplicative constant), the obtained constant $\mathcal{C}_{BDK}$ is sharp, that is $\mathcal{C}_{WGNS}=\mathcal{C}_{BDK}$, see~\cite[page 4]{Balogh_2025}. This was also previously observed by N. Lam in~\cite{Lam_2017} who considered a similar case with different weights.  On the other hand, Z. Balogh, S. Don, and A. Krist\'aly also provide a (conditional) rigidity result; see~\cite[Theorem 1.2]{Balogh_2025}, which roughly says that if $\mathcal{C}_{BDK}$ is sharp then the three weights must be equal up to a multiplicative constant. Let us be more specific: this rigidity result holds under conditions $(\textrm{H}_1)$ and $(\textrm{H}_2)$ which are stated below. The first one being: 
\begin{equation}\label{H_1}\tag{$\textrm{H}_1$}
\begin{array}{cc}
    \textit{There exists a non-zero function $F$ such that} \\ [2mm]
    \mathcal{C}_{BDK}\,\left(\int_E |F|^{p\sigma}\omega_1dz\right)^\frac{1}{q}=\left(\int_E |\nabla F|^p\omega_2 dz\right)^\frac{1-\theta}{p}\left(\int_E |F|^{p\sigma\gamma}\omega_3 dz\right)^\frac{1}{s}\,.
\end{array}
\end{equation}
In order to state $(\textrm{H}_2)$ we need the following inequality, which shall play a major role in our paper: 
\begin{equation}\label{WCL}\tag{WCL}
\mathcal{C}_{WCL}\left(\int_E \rho^\gamma \omega_1^\frac{1}{q}\omega_2^\frac{1}{p}dz\right)\leq \left(\int_E \rho \omega_1dz\right)^\frac{\gamma\beta}{\alpha+\beta}\left(\int_E \rho |z|^{q}\omega_1dz\right)^\frac{\gamma\alpha}{\alpha+\beta}\,,
\end{equation}
where $0<\gamma<1$, $q=p/(p-1)$, $0<\alpha, \beta$ are fixed by scaling, and $\mathcal{C}_{WCL}$ is the optimal constant. The second assumption for the result in~\cite[Theorem 1.2]{Balogh_2025} is:
\begin{equation}\label{H_2}\tag{$\textrm{H}_2$}
\begin{array}{cc}
\textit{There exists a non-zero optimal function for inequality~\eqref{WCL}.}
\end{array}
\end{equation}
The existence of an optimizer for inequality~\eqref{WCL} plays an important part in the proof of the rigidity result in~\cite[Theorem 1.2]{Balogh_2025}; however, the authors clarify, see~\cite[Comments after Theorem 1.2]{Balogh_2025}, that it is unclear whether this holds true for any triple of weights in their setting. The scope of the present paper is exactly to give a precise answer to this question under minimal integrability assumptions on the weights. \\
\newline

Inequality~\eqref{WCL} is known in the literature under the name of Carlson-Levin's inequality; see, for instance~\cite{Carrillo_2019}. In the simplest ($q=2$) and unweighted case, this inequality can be simply stated (when posed on $\mathbb{R}^d$) as follows:
\begin{equation}\label{simple-CL}
\mathcal{C}_{CL} \left(\int_{\mathbb{R}^d} \rho^\gamma dz \right)^\frac{1}{\gamma} \leq \left(\int_{\mathbb{R}^d} \rho\,|z|^2dz\right)^\frac{d(1-\gamma)}{2\gamma}\left(\int_{\mathbb{R}^d} \rho \,dz\right)^\frac{2\gamma-d(1-\gamma)}{2\gamma}\,,
\end{equation}
where $1>\gamma>\frac{d}{d+2}$. It has already been observed in~\cite[Remark (3), pag 318]{Cordero_Erausquin_2004} that it plays the role of a dual inequality for the Sobolev and Gagliardo-Nirenberg inequalities. In the unweighted case, it is known that the optimal functions (up to a multiplication by a constant and re-scaling) are Barenblatt profiles (also called Aubin-Talenti profiles), given by
\begin{equation}\label{optimizer-simple-case}
W(z)=\left(1+|z|^2\right)^\frac{1}{1-\gamma}\,.
\end{equation}
A proof of this fact can be found in~\cite[Lemma 5]{Carrillo_2019}, see also section~\ref{sec<1}.\\

Let us return to inequality~\eqref{WCL}: here we consider $E\subseteq \mathbb{R}^d$ an open convex cone and two homogeneous weights $\omega_1,\omega_2\in C(E)$ such that, for any $t>0$ we have
\begin{equation}\label{homogenuity}
\omega_i(tz)=t^{h_i}\omega_i(z)\quad\mbox{where}\quad h_1,h_2>-d\,.
\end{equation}
It is natural to consider weights that are integrable on the set $\mathbb{S}^{d-1}_E:=\mathbb{S}^{d-1}\cap E$ with respect to the Riemannian measure $dV_{\mathbb{S}^{d-1}}$ of the sphere $\mathbb{S}^{d-1}\subset\mathbb{R}^d$, eventually restricted to $\mathbb{S}^{d-1}_E$. 

Our first result concerns inequality~\eqref{WCL}, in the case $0<\gamma<1$, and, in its simplest form, can be stated as follows:
\begin{theorem}\label{first-result:simple}
Let $p\in\left(1,\infty\right)$, $q=p/(p-1)$. Assume furthermore that
\begin{equation}\label{thm-1.integrability-condition} 
\omega_1^{1+\frac{1}{p(\gamma-1)}}\omega_2^\frac{1}{p(1-\gamma)}\in \mathrm{L}^1(\mathbb{S}^{d-1}_E, dV_{\mathbb{S}^{d-1}})\,,
\end{equation}
and that
\begin{equation}\label{thm-1.necessary}
	\frac{d+\frac{h_1}{q}+\frac{h_2}{p}}{d+h_1+q}<\gamma<\frac{d+\frac{h_1}{q}+\frac{h_2}{p}}{h_1+d}\,.
\end{equation}
Then inequality~\eqref{WCL} holds and the optimal function is given by
\begin{equation}\label{thm-1.optimal}
W(z):=\omega_1^\frac{1}{p(\gamma-1)}(z)\,\omega_2^\frac{1}{p(1-\gamma)}(z)\,\left(C+|z|^{q}\right)^\frac{1}{\gamma-1}\quad\text{for some}\,\,C>0\,.
\end{equation}
\end{theorem}
Some remarks are in order. We first notice that Theorem~\ref{first-result:simple} gives an explicit and easy-to-verify condition for~\eqref{H_2} to hold. This characterisation has several applications to the weighted Gagliardo-Nirenberg inequalities~\eqref{eqGN} which we shall further explore in section~\ref{secgagnir}. As it will be clear from the proof, condition~\eqref{thm-1.necessary} is sufficient for inequality~\eqref{WCL} to hold. However, it is also a necessary condition for the optimizer $W$ to have finite weighted $q$-moment and weighted integrability. We also notice that the shape of the optimizer $W$ resembles the one of~\eqref{optimizer-simple-case}, with a corrective factor which depends on the weights. Assumption~\eqref{thm-1.integrability-condition} is more intriguing: it appears naturally in our proof but we do not have an intuitive explanation for it. When the weights $\omega_1 $ and $\omega_2$ are equal up to a constant ($\omega_1=\omega_2=\omega$ for some weight $\omega$), condition~\eqref{thm-1.integrability-condition} simply becomes $\omega \in \mathrm{L}^1(\mathbb{S}^{d-1}_{E}, dV_{\mathbb{S}^{d-1}})$, which seems quite reasonable. While assumption~\eqref{thm-1.integrability-condition} is sufficient for the existence of an optimizer (and therefore for~\eqref{H_2} to hold), it is also not far from being necessary. Indeed, if there exists an optimizer $\overline{W}$ for~\eqref{WCL} then (up to an integrability condition) assumption~\eqref{thm-1.integrability-condition} holds: this is contained in the following theorem.
\begin{theorem}\label{almost-optimality-condition-weights}
Assume that $\overline{W}$ is an optimizer for~\eqref{WCL} such that  $\int_{K}\overline{W}^{\gamma-1}(z)\,\omega_1^\frac{1}{q}(z)\,\omega_2^\frac{1}{p}(z)\,dz<\infty$ for any compact set $K\subset E$. Then
\begin{itemize}
	\item[i)] Up to a multiplicative constant $W=\overline{W}$, where $W$ is as in~\eqref{thm-1.optimal};
	\item[ii)] We have that $\omega_1^{1+\frac{1}{p(\gamma-1)}}\omega_2^\frac{1}{p(1-\gamma)}\in \mathrm{L}^1(\mathbb{S}^{d-1}_E, dV_{\mathbb{S}^{d-1}})$.
\end{itemize}
\end{theorem}
To conclude our remark about condition~\eqref{thm-1.integrability-condition}, we notice that it is possible to construct explicit examples of weights that do not satisfy assumption~\eqref{thm-1.integrability-condition}: we show in such a case that inequality~\eqref{WCL} does not hold; see Remark~\ref{rmkexample}. 

The condition $0<\gamma<1$ appears naturally in the proof of~\eqref{eqGN} both in the weighted case, see~\cite{Lam_2017,NGUYEN_2021,Balogh_2025}, or in the un-weighted case~\cite{Cordero_Erausquin_2004}. It is however possible to consider also the Carlson-Levin inequality when $\gamma>1$. In our setting, it amounts to the following: 
\begin{equation}\label{eq1.2}
	\mathcal{C}_{WCL}\left(\int_E \rho \omega_1dz\right)^{\gamma+\frac{\alpha}{q}}\leq \left(\int_E \rho^\gamma \omega_1^\frac{1}{q}\omega_2^\frac{1}{p}dz\right)\left(\int_E \rho |z|^{q}\omega_1dz\right)^\frac{\alpha}{q}\,, 
\end{equation}
where, as before, $p\in\left(1,\infty\right)$, $q=p/(p-1)$ and $\alpha$ is fixed by scaling. As in the case $\gamma<1$, inequality~\eqref{eq1.2} appears as a dual inequality for a family of weighted Gagliardo-Nirenberg inequalities, see~\cite[Theorem 1.3]{Balogh_2025}. Our main result for~\eqref{eq1.2} is similar to Theorem~\ref{first-result:simple}: we give necessary and sufficient condition for inequality~\eqref{eq1.2} to hold. In particular, we show that optimizer for~\eqref{eq1.2} are equal (up to a constant) to
\[
W(z)=\omega_1^\frac{1}{p(\gamma-1)}(z)\,\omega_2^\frac{1}{p(1-\gamma)}(z)\,\left(C-|z|^{q}\right)_+^\frac{1}{\gamma-1}\quad\text{for some}\,\,C>0\,,	
\]
see section~\eqref{sec>1}. The limit case $\gamma\to1$ can also be considered, which leads to the following inequality, also sometimes referred to as Shannon's inequality (see, for instance~\cite[Corollary 1.4]{Kubo_2018})
\begin{equation*}
		\frac{(d+h_1)}{\beta}\,\log\left(\mathrm{C}_W\,\frac{\beta e }{q}\int_E \frac{\rho\,|z|^q\,\omega_2}{\norm{\rho}_{\mathrm{L}^1_{\omega_1}(E)}}  dz\right)\norm{\rho}_{\mathrm{L}^1_{\omega_1}(E)}
		\geq -\int_E\rho\log\left(\frac{\rho}{\norm{\rho}_{\mathrm{L}^1_{\omega_1}(E)}}\right)\, \omega_1 dz\,,
\end{equation*}
where, in this case $\beta=q+h_2-h_1>0$, $\mathrm{C}_W=\int_{E}W\,\omega_1\,dz$ where $W$ is the optimal function given by
\[
	W(z)=\exp\left(-\frac{|z|^q}{q}\,\frac{\omega_2(z)}{\omega_1(z)}\right)\,;
\]
see Theorem~\ref{thm:gamma=1} for more details.

Finally, let us explain the main ideas of our method in order to prove Theorems~\ref{first-result:simple} and~\ref{almost-optimality-condition-weights}. Our intuition comes from the fast-diffusion equation\footnote{Perhaps the reader may be confused by the fact that the fast-diffusion equation is generally presented as $\partial_t u = \Delta u^m$ where $m\in\left(0,\infty\right)$. It is easy to see that the two versions are related by a simple change of variables, see for instance~\cite[section 2]{Bonforte_2025} or~\cite{VazquezBook2006,VazquezBook2007}}  which can be written as
\begin{equation}\label{FD}
\frac{\partial v}{\partial t}+\nabla\cdot\left(v\,\nabla v^{\gamma-1}\right)=2\,\nabla\cdot\left(z\,v\right)\,.
\end{equation}
In the case $\gamma<1$ and when posed on $\mathbb{R}^d$, it is not hard to see that the function $W(z)=\left(1+|z|^2\right)^\frac{1}{\gamma-1}$, defined in~\eqref{optimizer-simple-case}, is a stationary solution. Since the seminal papers of J. Ralston and W. I. Newmann~\cite{Ralston_1984,Newman_1984}, it is common knowledge that~\eqref{FD} admits an entropy functional\footnote{Roughly speaking, an entropy functional decreases in time if computed along the flow and controls the convergence to the equilibrium in the $L^1$ topology} that can be written for a function $v\ge0$ as
\[
	H[v|W]:=\frac{1}{\gamma-1}\int_{\mathbb{R}^d}\left(v^\gamma-W^\gamma-\gamma\,W^{\gamma-1}(v-W)\right)dz\,,
\]
where $W$ is as in~\eqref{optimizer-simple-case}. The main observation on which we built our proofs is the following (we assume that $\int_{\mathbb{R}^d}vdz=\int_{\mathbb{R}^d}Wdz$): the non-negativity of the entropy functional $H[v|W]\ge0$ is equivalent to the validity of inequality~\eqref{simple-CL}. From this observation it is possible to compute the optimal constant $\mathcal{C}_{CL}$ as well as to classify the optimizers. The same ideas applies to the weighted version of the Carlson-Levin inequality~\eqref{WCL}. 

Let us now conclude this introduction by establishing our notation and the plan of the paper. Let us recall that the functions $\omega_1,\omega_2:E\to\left[0,\infty\right)$ are non-negative, continuous ($\omega_i\in C(E)$ for $i=1,2$) and homogeneous of degree $h_i$ respectively, i.e.,~\eqref{homogenuity} holds, and $h_i>-d$ for $i=1,2$. The measure $dz$ shall denote the Lebesgue measure on $E\subset \mathbb{R}^d$. We shall denote by $\textrm{L}^p(E, \omega\,dz)$ the standard $\textrm{L}^p$ spaces of function defined on $E$ which are $p$-integrable with respect to the measure $\omega\,dz$. We notice that we shall also make use of $\mathrm{L}^\gamma$ spaces even when $\gamma<1$, while this is not necessarily a normed space, it will nevertheless play a role in our reasoning. In order to study inequality~\eqref{WCL} it is natural to introduce the following functional space
\begin{equation}\label{G}
\mathrm{G}^{p,\gamma}:=\{u:E\to\left(0, \infty\right): u\in\mathrm{L}^1(E, \omega_1dz))\cap \mathrm{L}^\gamma(E, \omega_1^\frac{1}{q}\omega_2^\frac{1}{p}dz)\quad\mbox{and}\quad \int_{E} u|z|^{q}\omega_1dz<+\infty\}
\end{equation}
where $\gamma \in (0,+\infty)$, $p\in(1,+\infty)$ and $q=\frac{p}{p-1}$. We notice that, if $\rho\in \mathrm{G}^{p,\gamma}$, all the integrals in~\eqref{WCL} are finite. 

We shall denote by $\mathbb{S}^{d-1}_E:=\mathbb{S}^{d-1}\cap E$ the intersection between the cone $E$ and the sphere, and by $dV_{\mathbb{S}^{d-1}}$ the Riemannian measure on the sphere (eventually restricted to $\mathbb{S}^{d-1}_E$). Lastly, we also define the quantity $c_{\omega_1,\omega_2}>0$ which shall often appear in the computations 
\begin{equation}\label{constant-on-sphere}
c_{\omega_1,\omega_2}:=\int_{\mathbb{S}^{d-1}_E}\omega_1^{1+\frac{1}{p(\gamma-1)}}\,\omega_2^\frac{1}{p(1-\gamma)}\,dV_{\mathbb{S}^{d-1}}\,.
\end{equation}

The paper is organized as follows. In section~\ref{sec:history} we detail a brief history of the Carlson-Levin inequality. 
In section \ref{sec<1} we prove Theorems~\ref{first-result:simple} and~\ref{almost-optimality-condition-weights} and explain in detail the role played by all the assumptions. In section \ref{sec>1} we state and prove the analogous results for the $\gamma>1$ case, while the limit case $\gamma\to1$ is treated in section~\ref{seclog}. Lastly, in section \ref{secgagnir} we show an application of our main result (Theorem~\ref{first-result:simple}) to the rigidity property of the weighted Gagliardo-Nirenberg inequality~\eqref{eqGN}, improving upon~\cite[Theorem 1.2]{Balogh_2025}.

\section{A brief history of the Carlson-Levin inequality and related results}\label{sec:history}
Inequality~\eqref{WCL} first appeared in literature in the work \cite{zbMATH03014034} by F. Carlson. At the beginning,  it had the form of an H\"older-type inequality for sequences, and then it was studied in the continuous setting under the Lebesgue measure. Thereafter, V.I. Levin in \cite{zbMATH03047408} proved its optimal form. \\
However, this inequality is called with many names in the literature, see for instance~\cite{Lutwak_2004} and the references in~\cite{Carrillo_2019}. It is also worth mentioning that similar inequalities are also treated in \cite{mitrinovic2012inequalities}.\\   \ \\
These inequalities have many applications. For instance, they imply the positivity of the relative entropy associated with a non-linear Fokker-Planck type equation, and they lead to a Heisenberg-type uncertainty principle. Also, they can show blow-up and concentration phenomena for some drift-diffusion equations. For more details see \cite{Kubo_2018,Ogawa_2022,Suguro_2022}. \\ \ \\
Concerning inequality \eqref{eqGN}, we remark that weighted Sobolev and Gagliardo Nirenberg inequalities have been extensively studied in recent years. For instance, X. Cabr\'e, X. Ros-Oton and J. Serra studied Sobolev inequalities with monomial weights and Isoperimetric ones with more general log concave weights, via the ABP method (see \cite{Cabr__2013,Cabr__2016}). Instead, an optimal transport approach for weighted isoperimetric inequalities can be found in \cite{Lam_2017} by N. Lam. Another important paper to mention is the one of G. Ciarolo, A. Figalli and A. Roncoroni \cite{Ciraolo_2020} where they proved weighted Sobolev inequalities for general norms in $\mathbb{R}^d$ still under a concavity condition on the weight. For the unweighted case and their connection with diffusion equations, see \cite{Del_Pino_2002, Dolbeault_2003} by M. Del Pino and J. Dolbeault
\\ \ \\
Unlike the aforementioned works, where the same weight is considered in every integral, other authors investigated weighted Sobolev inequalities with different weights in the $L^{p^*}$ norm of the function and the $L^2$ norm of the gradient. Apart from the more classical Caffarelli-Kohn-Nirenberg inequality (see \cite{caffarelli1984first}) more recent papers of this kind came out. An example is the work by H. Castro \cite{Castro_2016} where he considered Sobolev inequalities with different monomial weights. It is in this context that \cite[Theorem 1.2]{Balogh_2025} and Corollary \ref{corrig} can be applied to show that, if weights satisfy \eqref{eqmain}, then the inequalities cannot have an optimal function (unless the weights are equal). This is not the case for CKN weights, as a quick computation shows: under some hypotheses on the parameters, optimizers exist and are explicit.  \\ \ \\
At last, linked to our limiting case of $\gamma=1$ and of the Gagliardo-Nirenberg inequality, we remark that log-Sobolev inequalities have been extensively investigated too. In particular, in \cite{Balogh_2024}, again by using optimal transport theory, the authors proved the sharp version of a weighted $p$-log-sobolev inequality. They considered log-concave and homogeneous weights and completely classified optimizers for big dimensions, solving an open problem even for the unweighted case. 
%%%%%%%%%%%%%%%%%%%%%%%%%%%%%%%%%%%%%%%%%
%%%%%%%%%%%%%%%%%%%%%%%%%%%%%%%%%%%%%%%%%%

\section{The case \texorpdfstring{$0<\gamma<1$}{gamma<1}}\label{sec<1}
We recall that $\omega_1,\omega_2\in C(E)$ are non-negative and homogeneous weights with homogeneity $h_i>-d$ for any $i=1,2$; recall as well that $\mathrm{G}^{p,\gamma}$ is defined in~\eqref{G}. Lastly, we shall make use of the Beta function defined for any $z,w\in\mathbb{C}$ such that $\mathcal{R}e(z),\mathcal{R}e(w)>0$:
\begin{equation}\label{beta}
	B(z,w):=\int_0^1t^{z-1}(1-t)^{w-1}dt\,,
\end{equation}
see~\cite[section 6.2.1]{Handbook}.  Let us now state again our main theorem when $0<\gamma<1$ with more details.
\begin{theorem}\label{thm1}
Let $p\in\left(1,\infty\right)$, $q=p/(p-1)$ and $\gamma\in\left(0,1\right)$. 
Assume furthermore that 
\begin{align}
	&\omega_1^{1+\frac{1}{p(\gamma-1)}}\omega_2^\frac{1}{p(1-\gamma)}\in \mathrm{L}^1(\mathbb{S}^{d-1}_E)\label{hp1}\\
	&\gamma<\frac{d+\frac{h_1}{q}+\frac{h_2}{p}}{h_1+d}\label{hp2}\\
	&\gamma>\frac{d+\frac{h_1}{q}+\frac{h_2}{p}}{d+h_1+q}.\label{hp3}
\end{align}
Then~\eqref{WCL} is valid for all $\rho\in \mathrm{G}^{p,\gamma}$, the optimal constant being
\begin{align*}
    \mathcal{C}_{WCL} =\left( \frac{c_{\omega_1,\omega_2}}{q}\right)^{\gamma-1}(1+j\gamma -j)^{\frac{\alpha}{q}-1}(\gamma-j+\gamma j)^{\gamma-\frac{\alpha}{q}}\gamma^{-\gamma}  \frac{q}{\alpha}\left(\frac{\alpha}{q-\alpha}\right)^\frac{\alpha-q}{q} B\left( j;\frac{\gamma}{1-\gamma}-j\right)^{\gamma-1}\,,
\end{align*}
where
\begin{equation}\label{parameters-gamma-less-1}
    j:=\frac{d}{q}+\frac{h_1}{q}+\frac{h_1}{pq(\gamma-1)}+\frac{h_2}{pq(1-\gamma)},\quad \alpha:= d+\frac{h_1}{q}+\frac{h_2}{p}+\gamma(-d-h_1)\,,
\end{equation}
the function $B$ is as in~\eqref{beta}.  Furthermore, the function: 
\begin{equation}\label{eqoptimiser}
W(z):=\omega_1^\frac{1}{p(\gamma-1)}\omega_2^\frac{1}{p(1-\gamma)}(C+|z|^{q})^\frac{1}{\gamma-1}\quad\text{for some}\quad C>0\,,
\end{equation} 
is an optimizer for inequality~\eqref{WCL}.
\end{theorem}
Some remarks are in order. Let us begin with assumption~\eqref{hp1}: in Remark \ref{rmkexample} we show that such an assumption is a necessary one. Inequality~\eqref{WCL} cannot hold if~\eqref{hp1} is removed. In the case $\omega_1=\omega_2=\omega$, assumption~\eqref{hp1}  simply tells that $\omega\in \mathrm{L}^1(\mathbb{S}^{d-1}_E)$. This assumption is also necessary to the existence of an optimizer, which is proven in Theorem~\ref{thm1.1}. 

Let us also discuss the integrability properties of the optimizer $W$. Conditions\eqref{hp2} and~\eqref{hp3} are necessary for the integrability of $W$ at zero and at infinity and are heavily used in the proof of the equality itself. Assumption~\eqref{hp3} also implies that 
\begin{align}\label{hp3.1}
	1>\frac{d+\frac{h_1}{q}+\frac{h_2}{p}}{d+h_1+q}\Rightarrow q>-\frac{h_1}{p}+\frac{h_2}{p}\Rightarrow \frac{d+\frac{h_1}{q}+\frac{h_2}{p}}{d+h_1+q}>\frac{d+\frac{h_1}{q}+\frac{h_2}{p}-q}{d+h_1}\,.
\end{align} 
This chain of inequalities, together with \eqref{hp1}, \eqref{hp2}, \eqref{hp3} ensures that $W\in \mathrm{G}^{p,\gamma}$. 
\begin{proof}[Proof of Theorem \ref{thm1}]
The strategy for the proof of goes as follows. At first we prove, via H\"older's inequality, that~\eqref{WCL} holds for a constant $C>0$. We then find the optimal constant and (one of) the optimizers by relating~\eqref{WCL} to a (weighted) entropy functional similar to the functional $H$ defined in the introduction. \\
\textit{Step I: Inequality \eqref{WCL} holds for some constant $C>0$}. Let us begin by proving that inequality \eqref{WCL} holds, without obtaining the optimal constant. Our proof is inspired by~\cite[Lemma 5]{Carrillo_2019}. Fix $R>0$ and consider
\begin{align*}
	\int_{E}\rho^\gamma \omega_1^\frac{1}{q}\omega_2^\frac{1}{p}dz=\int_{E\cap B_R}\rho^\gamma \omega_1^\frac{1}{q}\omega_2^\frac{1}{p}dz+\int_{E\cap B_R^c}\rho^\gamma \omega_1^\frac{1}{q}\omega_2^\frac{1}{p}dz=:I_R+II_R.
\end{align*}
Then, by using H\"older's inequality on $I_R$ we get:
\begin{align*}
	I_R=\int_{E\cap B_R} \rho^\gamma \omega^\gamma_1\omega_1^{\frac{1}{q}-\gamma}\omega_2^\frac{1}{p}dz\leq \left(\int_{E\cap B_R} \rho \omega_1dz\right)^\gamma\left(\int_{E\cap B_R}\omega_1^{1+\frac{1}{p(\gamma-1)}}\omega_2^\frac{1}{p(1-\gamma)}dz\right)^{1-\gamma}\,.
\end{align*}
By integrating in spherical coordinates the term $\int_{E\cap B_R}\omega_1^{1+\frac{1}{p(\gamma-1)}}\omega_2^\frac{1}{p(1-\gamma)}dz$ we get
\begin{align*}
	\int_{E\cap B_R}\omega_1^{1+\frac{1}{p(\gamma-1)}}\omega_2^\frac{1}{p(1-\gamma)}dz&=\int_{0}^R r^{\frac{\alpha}{1-\gamma}-1}dr\int_{\mathbb{S}^{d-1}_E}\omega_1^{1+\frac{1}{p(\gamma-1)}}\omega_2^\frac{1}{p(1-\gamma)}dV_{\mathbb{S}^{d-1}}=\frac{c_{\omega_1,\omega_2}(1-\gamma)}{\alpha}\,R^{\frac{\alpha}{1-\gamma}}\,,
\end{align*}
where in the last inequality we have used~\eqref{hp1}, $c_{\omega_1,\omega_2}$ is as in~\eqref{constant-on-sphere} and $\alpha$ as in~\eqref{parameters-gamma-less-1}. We conclude that 
\begin{align*}
	I_R\leq \left(\frac{c_{\omega_1,\omega_2}(1-\gamma)}{\alpha}\right)^{1-\gamma}\left(\int_{E} \rho \omega_1dz\right)^\gamma R^{\alpha}=: A R^\alpha\,.
\end{align*}
Performing similar computations we can bound $II_R$ similarly
\begin{align*}
	II_R\leq \left(\frac{c_{\omega_1,\omega_2}(1-\gamma)}{\beta}\right)^{1-\gamma} \left(\int_{E}\rho \omega_1 |z|^{q}dz\right)^\gamma R^{-\beta}=:B\,R^{-\beta},
\end{align*}
with $-\beta:= d+\frac{h_1}{q}+\frac{h_2}{p}+\gamma(-d-h_1-q)<0$ from \eqref{hp3}. Using the above two estimates, we find that, for any $R>0$,
\begin{align*}
	\int_E \rho^\gamma \omega_1^\frac{1}{q}\omega_2^\frac{1}{p} dz\leq A\,R^\alpha+B\,R^{-\beta}:=f(R).
\end{align*}
By taking the infimum of $f$ on $\left[0,\infty\right)$ we obtain
\begin{align*}
	\int_E \rho^\gamma \omega_1^\frac{1}{q}\omega_2^\frac{1}{p}dz \leq C\,\left(\int_E \rho \omega_1dz\right)^\frac{\gamma\beta}{\alpha+\beta}\left(\int_E \rho \omega_1|z|^{q}dz\right)^\frac{\gamma\alpha}{\alpha+\beta}
\end{align*}
where 
\begin{align*}
	C=\left(\int_{\mathbb{S}^{d-1}_E}\omega_1^{1+\frac{1}{p(\gamma-1)}}\omega_2^\frac{1}{p(1-\gamma)}dV_{\mathbb{S}^{d-1}}\right)^{1-\gamma}\left(\left(\frac{\beta}{\alpha}\right)^\frac{\alpha}{\alpha+\beta}+\left(\frac{\alpha}{\beta}\right)^\frac{\beta}{\alpha+\beta}\right).
\end{align*}
\textit{Step II: Computation of the optimal constant.} The above procedure proves that inequality~\eqref{WCL} holds under the assumptions~\eqref{hp1},~\eqref{hp2} and~\eqref{hp3}; it does not, however, give us the optimal constant. For this purpose, let us define the entropy functional (inspired by~\cite[section 2]{Bonforte_2025} and references therein):
\begin{equation*}
	\mathcal{F}[\rho]:=\frac{1}{\gamma-1}\int_E \rho^\gamma  \omega_1^\frac{1}{q}\omega_2^\frac{1}{p}dz-\frac{\gamma}{\gamma-1}\int_E \rho |z|^{q}\omega_1 dz
\end{equation*}
for all $\rho\in \mathrm{G}^{p,\gamma}$ such that $\int_E \rho \omega_1 dz=\int_E W \omega_1 dz$, where $W$ is an in~\eqref{eqoptimiser}. The convexity of the function $x\mapsto\frac{x^\gamma}{\gamma-1}$ implies that 
\begin{equation}\label{eqconv}
	\frac{\rho^\gamma}{\gamma-1}-\frac{W^\gamma}{\gamma-1}-\gamma\frac{W^{\gamma-1}}{\gamma-1}(\rho-W)\geq 0\quad\mbox{a.e. in}\quad E\,.
\end{equation}
Integrating \eqref{eqconv} with respect to $\omega_1^\frac{1}{q}\omega_2^\frac{1}{p}dz$ and using the fact that $W^{\gamma-1}=\omega_1^\frac{1}{p}\omega_2^{-\frac{1}{p}}(C+|z|^q)$ and $\int_E \rho \omega_1 dz=\int_E W \omega_1 dz$, we obtain that the relative entropy function is non-negative
\begin{equation}\label{relative-entropy-gammma-less-1}
H[\rho\lvert W]:=\mathcal{F}[\rho]-\mathcal{F}[W]\geq 0\qquad\forall \rho\in \mathrm{G}^{p,\gamma}\quad\text{such that }\quad\int_E \rho\,\omega_1 dz=\int_E W\,\omega_1 dz.
\end{equation}  
Fix such a $\rho$ and define $\rho_s(z):=s^{d+h_1}\rho(sz)$ for any $s>0$. A simple computation shows that $\int_E \rho_s \omega_1 dz=\int_E \rho \omega_1 dz$. The entropy functional could be easily computed on the function $\rho_s$
\begin{align*}
\mathcal{F}[\rho_s]&=\frac{1}{\gamma-1}s^{(h_1+d)\gamma-d-\frac{h_1}{q}-\frac{h_2}{p}}\int_E \rho^\gamma  \omega_1^\frac{1}{q}\omega_2^\frac{1}{p}dz-\frac{\gamma}{\gamma-1}s^{-q}\int_E \rho |z|^{q}\omega_1 dz\\
&=\frac{1}{\gamma-1}s^{-\alpha}K_\rho-\frac{\gamma}{\gamma-1}s^{-q}H_\rho=\frac{1}{\gamma-1}s^{-q}(s^{q-\alpha}K_\rho-\gamma H_\rho)=:g_\rho(s)\,
\end{align*}
where $K_\rho=\int_E \rho^\gamma  \omega_1^\frac{1}{q}\omega_2^\frac{1}{p}dz$ and $H_\rho=\int_E \rho |z|^{q}\omega_1 dz$. Remark that \eqref{hp3} implies that $q-\alpha>0$. A simple computation shows that the function $g_\rho(s)$  has its minimum at $s^*:=\left(q\gamma H_\rho/\alpha K_\rho\right)^\frac{1}{q-\alpha}$ and that 
\begin{equation*}
0>g_\rho(s^*)=\mathcal{F}[\rho_{s^*}]=\frac{1}{\gamma-1}\left[\frac{1}{\gamma^\frac{\alpha}{q}}\frac{\alpha}{q}\left(\frac{q-\alpha}{\alpha}\right)^\frac{q-\alpha}{q}K_\rho H_\rho^{-\frac{\alpha}{q}}\right]^\frac{q}{q-\alpha}\geq \mathcal{F}[W]\,,
\end{equation*}
where the last inequality is implied by~\eqref{relative-entropy-gammma-less-1}. 
By considering that $\mathcal{F}[\rho]=g_\rho(1)\ge g_\rho(s^\star)$ we obtain:
\begin{equation*}
\frac{\int_E \rho^\gamma  \omega_1^\frac{1}{q}\omega_2^\frac{1}{p}dz}{\left(\int_E \rho |z|^{q}\omega_1 dz\right)^\frac{\alpha}{q}}\leq\left[(\gamma-1)\mathcal{F}[W]\right]^\frac{q-\alpha}{q}\gamma^\frac{\alpha}{q}\frac{q}{\alpha}\left(\frac{\alpha}{q-\alpha}\right)^\frac{q-\alpha}{q}=:M\,.
\end{equation*}
Repeating the argument with $W$ in place of $\rho$ we obtain
\begin{align*}
\frac{\int_E W^\gamma  \omega_1^\frac{1}{q}\omega_2^\frac{1}{p}dz}{\left(\int_E W |z|^{q}\omega_1 dz\right)^\frac{\alpha}{q}}=M\,,
\end{align*}
deducing that 
\begin{equation*}
\mathcal{C}_{WCL}=M^{-1}\left(\int_E W \omega_1 dz\right)^{\gamma-\frac{\alpha}{q}}
\end{equation*}
is the optimal constant in \eqref{WCL} and attained by $W$ itself.
Let us now compute $\mathcal{C}_{WCL}$ explicitly. Without loss of generality, we can assume $W(z)=\omega_1^\frac{1}{p(\gamma-1)}\omega_2^\frac{1}{p(1-\gamma)}(1+|z|^{q})^\frac{1}{\gamma-1}$.
Then,
\begin{align*}
    \int_E W^\gamma  \omega_1^\frac{1}{q}\omega_2^\frac{1}{p}dz&=c_{\omega_1,\omega_2}\int_0^{\infty}(1+r^{q})^\frac{\gamma}{\gamma-1}r^{d-1+h_1+\frac{h_1}{p(\gamma-1)}+\frac{h_2}{p(1-\gamma)}}dr\\&= \frac{c_{\omega_1,\omega_2}}{q}\int_{0}^{\infty}(1+t)^\frac{\gamma}{\gamma-1}t^{\frac{d}{q}+\frac{h_1}{q}+\frac{h_1}{pq(\gamma-1)}+\frac{h_2}{pq(1-\gamma)}-1}dt\\
    &= \frac{c_{\omega_1,\omega_2}}{q}B\left(j;\frac{\gamma}{1-\gamma}-j\right)\,,
\end{align*}
where in the last identity we have used the change of variables $s=t/(1+t)$. Similarly, we find that
\begin{equation}
\int_E W |z|^{q}\omega_1 dz=\frac{c_{\omega_1,\omega_2}}{q}B\left(1+j;\frac{\gamma}{1-\gamma}-j\right)\quad\mbox{and}\quad\int_E W \omega_1 dz=\frac{c_{\omega_1,\omega_2}}{q} \left(\frac{\gamma-j+\gamma j}{\gamma}\right)B\left( j;\frac{\gamma}{1-\gamma}-j\right)\,.
\end{equation}
Therefore, we find that
\begin{align*}
\mathcal{F}[W]&=c_{\omega_1,\omega_2}\left[\frac{1}{(\gamma-1)q}B\left(j;\frac{\gamma}{1-\gamma}-j\right)-\frac{\gamma}{(\gamma-1)q}B\left(1+j;\frac{\gamma}{1-\gamma}-j\right)\right]\\
&=\frac{c_{\omega_1,\omega_2}}{q}\left( \frac{1-j+\gamma j}{\gamma-1}\right)B\left(j;\frac{\gamma}{1-\gamma}-j \right).
\end{align*}
where in the last inequality we used the property of the Beta function $B(x+1;y)=\frac{x}{x+y}B(x;y)$. Remark that \eqref{hp3} and \eqref{hp3.1} imply that $\frac{1}{\gamma-1}+j<0$. The proof is concluded.
\end{proof}
The next part of the section is devoted to showing that assumption~\eqref{hp1} of local integrability of weights is indeed necessary for the validity of\eqref{WCL} and the existence of optimizers. This can be shown by simply taking $\omega_1$ and $\omega_2$ as monomial weights. For this purpose, define the quotient 
 \begin{equation}\label{quotient}
	\mathcal{Q}_{\omega_1, \omega_2}:=\sup_{\substack{\rho\in\mathrm{G}^{p,\gamma}\\\int_{E}\rho\omega_1dz=1}} \frac{\int_{E} \rho^\gamma  \omega_1^\frac1{q}\omega_2^\frac1{p}dz}{
\left(\int_{E} \rho|z|^{q}\omega_1dz\right)^\frac{\alpha}{q}
}.
\end{equation}
\begin{remark}\label{rmkexample} \normalfont Via a counterexample, let us point out the necessity of hypothesis \eqref{hp1}.\\ Let $p=2$, $d\ge3$, $\gamma\in\left(0,1\right)$ and $E=\{z=(z_1, \cdots, z_d)\in\mathbb{R}^d:z_1>0\}$. Set $\omega_1(z):=z_1^{-2}$, $\omega_2(z):=z_1^{-2+\sigma}$ with 
	\begin{equation*}
		0\leq\sigma<2\min\{\gamma,1-\gamma\}.
	\end{equation*}
	Then, $$\omega_1^\frac{2\gamma-1}{2(\gamma-1)}\omega_2^\frac{1}{2(1-\gamma)}=z_1^{-2+\frac{\sigma}{2(1-\gamma)}}\notin \mathrm{L}^1(\mathbb{S}^{d-1}_E)$$ since, by assumption we have $-2+\frac{\sigma}{2(1-\gamma)}<-1$. Consider $\phi\in C^\infty_c(B_2)$ such that $\phi=1$ in $B_1$ and $\phi\geq 0$ in $B_2$. Fix $\delta$ such that 
	\begin{equation*}
		\frac{\sigma}{2}< \delta <\min\{\gamma,1-\gamma\} 
	\end{equation*}
	and $\varepsilon\geq0$. Define
	\begin{align*}
		\psi_\varepsilon:=\phi ^\frac{1}{\gamma}z_1^\frac{1+\delta}{\gamma}(z_1^2+\varepsilon^2)^{-\frac{\delta}{\gamma}},\qquad c_\varepsilon:=\int_E\psi_\varepsilon \omega_1 dz <+\infty, \qquad u_\varepsilon(z):=\frac{\psi_\varepsilon(z)}{c_\varepsilon}.
	\end{align*}
	Remark that $u_\varepsilon\in \mathrm{L}^1(E,|z|^2\omega_1dz)\cap \mathrm{L}^1(E,\omega_1dz)$ for any $\varepsilon\geq 0$ and $c_\varepsilon$ has an upper-bound independent of $\varepsilon$. Also, $u_\varepsilon\in \mathrm{L}^\gamma(E,\omega_1^\frac{1}{2}\omega_2^\frac{1}{2}dz)$ for any $\varepsilon>0$ but $u_0\notin \mathrm{L}^\gamma(E,\omega_1^\frac{1}{2}\omega_2^\frac{1}{2}dz)$. By considering the sequence $\{u_\varepsilon\}$ as $\varepsilon\rightarrow 0$ we obtain $\mathcal{Q}_{\omega_1, \omega_2}=+\infty$ in \eqref{quotient} (by monotone convergence theorem). Therefore, inequality \eqref{WCL}  cannot hold without hypothesis \eqref{hp1}.
\end{remark}
We conclude the section proving the uniqueness of optimizers under an integrability condition. At first, let us briefly observe the relation among optimizers for the relative entropy $H[\cdot\,\lvert W]$ and those for inequality~\eqref{WCL}.
\begin{remark}\label{remark-optimisers} \normalfont
Fix $W$ as defined in \eqref{eqoptimiser}. If $\rho\in  \mathrm{G}^{p,\gamma}$ such that $\int_E \rho \omega_1 dz=\int_E W \omega_1 dz$ and $H[\rho\lvert W]=0$ then $\rho$ is an optimal function for \eqref{WCL}. Conversely, if $\rho\in  \mathrm{G}^{p,\gamma}$ such that $\int_E \rho \omega_1 dz=\int_E W \omega_1 dz$ is an optimal function for \eqref{WCL} then $H[\rho_{s^*}\lvert W]=0$, where $\rho_{s^*}$ is defined as in the proof of Theorem \ref{thm1}.
 \end{remark}
\begin{theorem}\label{thm1.1}
Suppose that \eqref{hp2} and \eqref{hp3} hold and that  $\overline{W}\in \mathrm{G}^{p,\gamma}$ is an optimizer for~\eqref{WCL} such that  $\int_{K}\overline{W}^{\gamma-1}(z)\,\omega_1^\frac{1}{q}(z)\,\omega_2^\frac{1}{p}(z)\,dz<\infty$ for any compact set $K\subset E$.
\begin{itemize}
	\item[i)] Up to a multiplicative constant $\overline{W}=W$, where $W$ is as in~\eqref{thm-1.optimal};
	\item[ii)] We have that $\omega_1^{1+\frac{1}{p(\gamma-1)}}\omega_2^\frac{1}{p(1-\gamma)}\in \mathrm{L}^1(\mathbb{S}^{d-1}_E, dV_{\mathbb{S}^{d-1}})$.
\end{itemize}
\end{theorem}

\begin{proof}[Proof of Theorem \ref{thm1.1}]
 Call $\overline{W}$ an optimizer of \eqref{WCL} and without loss of generality suppose that $\int_E \overline{W} \omega_1dz=1$. 
Fix $\delta >0$ and define
\begin{equation*}
	A_\delta:=\{z\in E: \overline{W}(z)\geq \delta \}.
\end{equation*}
Observe that, since $\overline{W}\neq0$ there exists at least a small $\delta>0$ for which $A_\delta \neq \emptyset$.
Take $\varphi\in \mathrm{G}^{p,\gamma}\cap L^\infty$ such that $\norm{\varphi}_{L^\infty}\leq 1$ and $\varphi=0$ in $A_\delta^c$. Therefore, the function $\overline{W}+\varepsilon \varphi$ is non-negative for $\varepsilon\in (-\delta,\delta)$. Since by hypothesis, up to a rescaling, $\overline{W}$ minimizes the entropy $\mathcal{F}[\cdot]$, we get by Fermat's theorem
\begin{align*}
	\frac{d}{d\varepsilon}\mathcal{F}[\psi_\varepsilon]\lvert_{\varepsilon=0}=0, \quad\text{where }\psi_\varepsilon:=\frac{\overline{W}+\varepsilon \varphi}{\int_{E}(\overline{W}+\varepsilon \varphi)\omega_1 dz}.
\end{align*}
Performing the computations we get 
\begin{equation*}
	\int_E \varphi[\omega_1(c_{\overline{W}}-|z|^{q})+\overline{W}^{\gamma-1}\omega_1^\frac{1}{q}\omega_2^\frac{1}{p}]dz=0\quad\text{where $c_{\overline{W}}$ is a real constant depending on $\overline{W}$}.
\end{equation*}
If we then impose $\varphi=\chi_{B_r(\bar{z})\cap A_\delta}$ for $\bar{z}\in A_\delta$ and $r$ such that $B_r(\bar{z})\subseteq E$, we obtain
\begin{equation*}
    \int_{B_r(\bar{z})}\chi_{\{\overline{W}\geq \delta\}}[\omega_1(c_{\overline{W}}-|z|^{q})+\overline{W}^{\gamma-1}\omega_1^\frac{1}{q}\omega_2^\frac{1}{p}]dz=0.
\end{equation*}
Therefore, Lebesgue's differentiation theorem says that
\begin{equation*}
    \overline{W}^{\gamma-1}=\omega_1^\frac{1}{p}\omega_2^{-\frac{1}{p}}(-c_{\overline{W}}+|z|^{q})\quad\text{a.e. in }\{\overline{W}>0\}.
\end{equation*}
This last equality implies that $c_{\overline{W}}<0$ otherwise $\overline{W}\notin L^1(E,\omega_1dz)$. As a consequence,
\begin{equation*}
    \overline{W}=\omega_1^\frac{1}{p(\gamma-1)}\omega_2^\frac{1}{p(1-\gamma)}(C+|z|^{q})^\frac{1}{\gamma-1}\quad\text{for }C>0
\end{equation*}
and the validity of \eqref{hp1} follows since $\overline{W}\in \mathrm{G}^{p,\gamma} $.
\end{proof}
%%%%%%%%%%%%%%%%%%%%%%%%%%%%%%%%%%%%%%%%%
%%%%%%%%%%%%%%%%%%%%%%%%%%%%%%%%%%%%%%%%%%
\section{The case \texorpdfstring{$\gamma>1$}{gamma>1}}\label{sec>1}
We begin by stating with full details our main theorem in this case.
\begin{theorem}\label{thm2}
Let $p\in\left(1,\infty\right)$, $q=p/(p-1)$, $\gamma>1$, and
\begin{equation}\label{second-condition}
\omega_1^{1+\frac{1}{p(\gamma-1)}}\omega_2^\frac{1}{p(1-\gamma)}\in \mathrm{L}^1(\mathbb{S}^{d-1}_E, dV_{\mathbb{S}^{d-1}})\quad\mbox{and}\quad \gamma > \frac{d+\frac{h_1}{q}+\frac{h_2}{p}}{d+h_1}\,.
\end{equation}
Then~\eqref{eq1.2} is valid for all $\rho\in \mathrm{G}^{p,\gamma}$ and, up to a multiplicative constant, all the optimizers have the following form
\begin{equation*}
	W(z)=\omega_1^\frac{1}{p(\gamma-1)}\omega_2^\frac{1}{p(1-\gamma)}(C-|z|^{q})_+^\frac{1}{\gamma-1}\quad\text{for some}\quad C>0.
\end{equation*}  
The optimal constant in~\eqref{eq1.2} is equal to
\begin{equation*}
    \mathcal{C}_{WCL}=\left(\frac{c_{\omega_1,\omega_2}}{q}\right)^{1-\gamma}\gamma \left[\left(\frac{q}{\alpha}\right)^\frac{\alpha}{q+\alpha}+\left(\frac{\alpha}{q}\right)^\frac{q}{q+\alpha}\right]^{-\frac{q+\alpha}{q}} B\left( j;\frac{\gamma}{\gamma-1}\right)^{1-\gamma}
\end{equation*}
where
\begin{equation}\label{parameters-gamma-bigger-1}
    j:=\frac{d}{q}+\frac{h_1}{q}+\frac{h_1}{pq(\gamma-1)}+\frac{h_2}{pq(1-\gamma)},\quad\alpha:=(h_1+d)\gamma-d-\frac{h_1}{q}-\frac{h_2}{p}.
\end{equation}
\end{theorem}
\begin{remark} \normalfont
By using similar techniques as in the proof of Theorem~\ref{thm1} it is possible to prove that, if an optimizer exists in $\mathrm{G}^{p, \gamma}$, then the weights must satisfy the integrability assumption $\omega_1^{1+\frac{1}{p(\gamma-1)}}\omega_2^\frac{1}{p(1-\gamma)}\in \mathrm{L}^1(\mathbb{S}^{d-1}_E, dV_{\mathbb{S}^{d-1}})$. 
\end{remark}
\begin{proof}[Proof of Theorem \ref{thm2}] The proof is similar to the one of Theorem~\ref{thm1}. The main ideas for the computation of the optimal constant go back to~\cite[Lemma 4.4]{carrillo2000asymptotic}

\noindent\textit{Step I: inequality~\eqref{eq1.2} holds for some constant $C>0$}.
Set $R>0$ and consider
\begin{align*}
	\int_E \rho \omega_1 dz=\int_{E\cap B_R} \rho \omega_1 dz+\int_{E\cap B_R^c} \rho \omega_1 dz=I_R+II_R\,.
\end{align*}
By using H\"older's inequality on $I_R$ and by integrating in spherical coordinates, we get that 
\begin{equation*}
I_R\leq A\, R^\alpha\quad\mbox{where}\quad A:=\left(\frac{(\gamma-1)c_{\omega_1,\omega_2}}{\alpha}\right)^\frac{\gamma-1}{\gamma} \left(\int_E \rho^\gamma \omega_1^\frac{1}{q}\omega_2^\frac{1}{p}dz\right)^\frac{1}{\gamma}\,,
\end{equation*}
where $\alpha$ is as in~\eqref{parameters-gamma-bigger-1}. Similarly, we obtain 
\begin{equation*}
II_R\leq B\, R^{-q}\quad\mbox{where}\quad\quad B:=\int_{E} \rho |z|^{q}\omega_1 dz\,.
\end{equation*}
We obtain that, for any $R>0$, we have $I_R+II_R\leq f(R)$ where $f(R):= A\,R^\frac{\alpha}{\gamma}+B\,R^{-q}$. As previously, by optimizing in $R$, we prove the validity of~\eqref{eq1.2}. 

\noindent\textit{Step II: non-negativity of the relative entropy functional.} 
In order to compute the optimal constant, as in the proof of Theorem~\ref{thm1}, we introduce the relative entropy functional. However, contrary to the case $\gamma<1$, here we need to work more for proving the non-negativity of the functional.

Fix $C>0$ in the definition of $W$ and consider $\rho \in \mathrm{G}^{p,\gamma}$ such that $\int_E \rho \omega_1dz=\int_E W\omega_1dz$. Define the entropy
\begin{align*}
	\mathcal{F}[\rho]:=\frac{1}{\gamma}\int_E \rho^\gamma \omega_1^\frac{1}{q}\omega_2^\frac{1}{p}dz+\int_E\rho |z|^{q}\omega_1  dz\,,
\end{align*}
and the relative entropy
\begin{align*}
	H[\rho\vert W]:=\mathcal{F}[\rho]-\mathcal{F}[W]\,.
\end{align*}
We prove that $H[\rho\vert W]\geq 0$ for such $\rho$. By the convexity of the function $x\mapsto x^\gamma$, we have
\begin{equation*}
	\frac{\rho^\gamma}{\gamma}-\frac{W^\gamma}{\gamma}- W^{\gamma-1}(\rho-W)\geq 0\quad\mbox{a.e. on the set}\,\,\,E\cap \{|z|^{q}<C\}\,.
\end{equation*}
By integrating this last expression on $E\cap \{|z|^{q}<C\}$ with respect to the measure $\omega_1^\frac{1}{q}\omega_2^\frac{1}{p} dz$, and by observing that $W^{\gamma-1}=(C-|z|^q)_+$ we get
\begin{equation}\label{inequality-1}
	\frac{1}{\gamma}\int_{E\cap \{|z|^{q}<C\}} (\rho^\gamma -W^\gamma)\omega_1^\frac{1}{q}\omega_2^\frac{1}{p}dz+\int_{E\cap \{|z|^{q}<C\}}(\rho-W)|z|^{q}\omega_1dz\geq C\int_{E\cap \{|z|^{q}<C\}}(\rho-W)\omega_1 dz.
\end{equation}
Also, since $\rho$ and $W$ have equal mass and the support of $W$ is in $\{|z|^{q}<C\}$, 
\begin{equation}\label{identity-1}
	\int_{E\cap \{|z|^{q}\geq C\}} \rho \omega_1 dz=-\int_{E\cap \{|z|^{q}< C\}} (\rho-W)\omega_1 dz.
\end{equation}
By using the previous inequality and identity we obtain
\begin{multline*}
	H[\rho\vert W]=\frac{1}{\gamma}\int_{E\cap \{|z|^{q}<C\}} (\rho^\gamma -W^\gamma)\omega_1^\frac{1}{q}\omega_2^\frac{1}{p}dz+\int_{E\cap \{|z|^{q}<C\}}(\rho-W)|z|^{q}\omega_1 dz\\
	\quad\quad+\frac{1}{\gamma}\int_{E\cap \{|z|^{q}\geq C\}}\rho^\gamma\omega_1^\frac{1}{q}\omega_2^\frac{1}{p}dz+\int_{E\cap \{|z|^{q}\geq C\}}\rho |z|^{q}\omega_1dz\\
	\geq C\int_{E\cap \{|z|^{q}<C\}}(\rho-W)\omega_1dz +\frac{1}{\gamma}\int_{E\cap \{|z|^{q}\geq C\}}\rho^\gamma\omega_1^\frac{1}{q}\omega_2^\frac{1}{p}dz+\int_{E\cap \{|z|^{q}\geq C\}}\rho |z|^{q}\omega_1dz\\
	= -C\int_{E\cap \{|z|^{q}\geq C\}} \rho \omega_1 dz+\frac{1}{\gamma}\int_{E\cap \{|z|^{q}\geq C\}}\rho^\gamma\omega_1^\frac{1}{q}\omega_2^\frac{1}{p}dz+\int_{E\cap \{|z|^{q}\geq C\}}\rho |z|^{q}\omega_1dz\\
	= \int_{E\cap \{|z|^{q}\geq C\}} \rho(|z|^{q}-C) \omega_1dz +\frac{1}{\gamma}\int_{E\cap \{|z|^{q}\geq C\}}\rho^\gamma\omega_1^\frac{1}{q}\omega_2^\frac{1}{p}dz\geq 0\,,
\end{multline*}
where in the second line we have used inequality~\eqref{inequality-1} and in the third line identity~\eqref{identity-1}.

\noindent\textit{Step III: Computation of the optimal constant.} We can now compute the optimal constant; the proof is very similar to the one of Theorem~\ref{thm1}. Define as before $\rho_s(z)=s^{d+h_1} \rho(sz)$ for $s>0$ (recall that $\int_{E} \rho_s \omega_1 dz = \int_{E} \rho \omega_1 dz$). The entropy functional could be easily computed as well in this case:
\begin{equation*}
\begin{split}
\mathcal{F}[\rho_s]&=\frac{1}{\gamma}s^{(h_1+d)\gamma-d-\frac{h_1}{q}-\frac{h_2}{p}}\int_E \rho^\gamma  \omega_1^\frac{1}{q}\omega_2^\frac{1}{p}dz+s^{-q}\int_E \rho |z|^{q}\omega_1 dz\\
&:=\frac{1}{\gamma}s^{\alpha}K_\rho+s^{-q}H_\rho=:g_\rho(s)\,,
\end{split}
\end{equation*}
where $K_\rho=\int_E \rho^\gamma  \omega_1^\frac{1}{q}\omega_2^\frac{1}{p}dz$ and $H_\rho=\int_E \rho |z|^{q}\omega_1 dz$. A simple computation shows that $g_\rho$ has its minimum at $s^*=\left(q\gamma H_\rho/\alpha K_\rho\right)^\frac{1}{q+\alpha}$ and that
\begin{equation*}
g_\rho(s^*)=\mathcal{F}[\rho_{s^*}]=\left( \frac{1}{\gamma}\left[\left(\frac{q}{\alpha}\right)^\frac{\alpha}{q+\alpha}+\left(\frac{\alpha}{q}\right)^\frac{q}{q+\alpha}\right]^\frac{q+\alpha}{q}K_\rho H_\rho^\frac{\alpha}{q}\right)^\frac{q}{q+\alpha}\geq \mathcal{F}[W].
\end{equation*}
This last inequality implies that
\begin{equation*}
\left(\int_E \rho^\gamma \omega_1^\frac{1}{q}\omega_2^\frac{1}{p}dz\right)\left( \int_E \rho|z|^{q}\omega_1dz\right)^\frac{\alpha}{q}\geq \gamma \left(\mathcal{F}[W]\right)^\frac{q+\alpha}{q} \left[\left(\frac{q}{\alpha}\right)^\frac{\alpha}{q+\alpha}+\left(\frac{\alpha}{q}\right)^\frac{q}{q+\alpha}\right]^{-\frac{q+\alpha}{q}}=:m.
\end{equation*}
Repeating the argument with $W$ in place of $\rho$ we obtain
\begin{align*}
\left(\int_E W^\gamma \omega_1^\frac{1}{q}\omega_2^\frac{1}{p}dz\right)\left( \int_E W|z|^{q}\omega_1dz\right)^\frac{\alpha}{q}=m
\end{align*}
deducing that
\begin{equation*}
\mathcal{C}_{CWL}=m\left(\int_E W \omega_1 dz\right)^{-\frac{\alpha}{q}-\gamma}
\end{equation*}
is the optimal constant in~\eqref{eq1.2} and attained by $W$ itself. Let us now explicitely compute $\mathcal{C}_{CWL}$. As in Theorem~\ref{thm1} we consider, in order to simplify the computations, that $W(z)=\omega_1^\frac{1}{p(\gamma-1)}\omega_2^\frac{1}{p(1-\gamma)}(1-|z|^{q})_+^\frac{1}{\gamma-1}$. As in the proof of Theorem~\eqref{thm1} we have the following identities
\begin{equation*}
\begin{split}
\int_E W^\gamma \omega_1^\frac{1}{q}\omega_2^\frac{1}{p}dz& =\frac{c_{\omega_1,\omega_2}}{q} B\left(j;\frac{\gamma}{\gamma-1}+1\right)\,,
\int_E W|z|^{q}\omega_1dz=\frac{c_{\omega_1,\omega_2}}{q} B\left(1+j;\frac{\gamma}{\gamma-1}\right)\,,\\
& \quad\mbox{and}\quad \int_E W\omega_1dz=\frac{c_{\omega_1,\omega_2}}{q}B\left( j;\frac{\gamma}{\gamma-1}\right)\,,
\end{split}
\end{equation*}
where $j$ is as in~\eqref{parameters-gamma-bigger-1}. Therefore,
\begin{equation*}
\mathcal{F}[W]=c_{\omega_1,\omega_2}\left[\frac{1}{\gamma q}B\left(j;\frac{\gamma}{\gamma-1}+1\right)+\frac{1}{q}B\left(1+j;\frac{\gamma}{\gamma-1}\right)\right]
=\frac{c_{\omega_1,\omega_2}(1+j)}{q(\gamma+j\gamma-j)}B\left( j;\frac{\gamma}{\gamma-1}\right)\,.
\end{equation*}

\noindent\textit{Step IV: Uniqueness of the optimizers.} To conclude, we prove uniqueness of the optimizers via the Euler-Lagrange equation they satisfy. Call $\overline{W}$ an optimizer for~\eqref{eq1.2} under the current assumptions.  Fix $\delta >0$ and define
\begin{equation*}
	A_\delta:=\{z\in E: \overline{W}(z)\geq \delta \}.
\end{equation*}
Observe that, since $\overline{W}\neq0$ there exists at least a small $\delta>0$ for which $A_\delta \neq \emptyset$.
Take $\varphi\in \mathrm{G}^{p,\gamma}\cap L^\infty$ such that $\norm{\varphi}_{L^\infty}\leq 1$ and $\varphi=0$ in $A_\delta^c$. Therefore, the function $\overline{W}+\varepsilon \varphi$ is non-negative for $\varepsilon\in (-\delta,\delta)$. Without loss of generality suppose that $\int_E \overline{W} \omega_1dz=1$ and since by hypothesis $\overline{W}$ minimizes the entropy $\mathcal{F}[\cdot]$ (see Remark~\ref{remark-optimisers}), we get by Fermat's theorem
\begin{equation*}
	\frac{d}{d\varepsilon}\mathcal{F}[\psi_\varepsilon]\lvert_{\varepsilon=0}=0, \quad\text{where }\psi_\varepsilon:=\frac{\overline{W}+\varepsilon \varphi}{\int_{E}(\overline{W}+\varepsilon \varphi)\omega_1dz}.
\end{equation*}
Performing the computations we get 
\begin{equation*}
	\int_E \varphi[\omega_1(|z|^{q}-c_{\overline{W}})+\overline{W}^{\gamma-1}\omega_1^\frac{1}{q}\omega_2^\frac{1}{p}]dz=0\quad\text{where $c_{\overline{W}}$ is a positive constant depending on $\overline{W}$}
\end{equation*}
If we then impose $\varphi=\chi_{B_r(\bar{z})\cap A_\delta}$ for $\bar{z}\in A_\delta$ and $r$ such that $B_r(\bar{z})\subseteq E$, we obtain
\begin{equation*}
	\int_{B_r(\bar{z})} \chi_{\{\overline{W}\geq \delta\}}[\omega_1(|z|^{q}-c_{\overline{W}})+\overline{W}^{\gamma-1}\omega_1^\frac{1}{q}\omega_2^\frac{1}{p}]=0.
\end{equation*}Lebesgue's differentiation theorem implies that
\begin{equation*}
	\chi_{\{\overline{W}\geq \delta\}}[\omega_1(|z|^{q}-c_{\overline{W}})+
    \overline{W}^{\gamma-1}\omega_1^\frac{1}{q}\omega_2^\frac{1}{p}]=0,\qquad \text{a.e. in }E
\end{equation*}
giving in turn (since the latter is valid for all $\delta>0$ sufficiently small) that
$\overline{W}=\omega_1^\frac{1}{p(\gamma-1)}\omega_2^\frac{1}{p(1-\gamma)}(c_{\overline{W}}-|z|^{q})_+^\frac{1}{\gamma-1}$. The proof is completed.
\end{proof}
%%%%%%%%%%%%%%%%%%%%%%%%%%%%%%%%%%%%%%%%%
%%%%%%%%%%%%%%%%%%%%%%%%%%%%%%%%%%%%%%%%%%
\section{The case \texorpdfstring{$\gamma=1$}{gamma=1}}\label{seclog}	
We complete our study by considering the limit case $\gamma=1$. In the present section we shall denote by $\mathrm{L}^1_{\omega_1}(E)$ the space $\mathrm{L}^1(E, \omega_1 dz)$. We also define, for any $q\ge1$
\begin{equation*}
	\mathrm{G}^{q,1}:=\{u:E\to\left(0, \infty\right): u\in\mathrm{L}^1_{\omega_1}(E)\quad\mbox{and}\quad \int_{E} u|z|^{q}\omega_2dz<+\infty, \int_E u\log (u)\, \omega_1dz<+\infty\}.
\end{equation*}
We also define the function
\begin{equation*}
	W:=\exp\left(-\frac{|z|^q}{q}\frac{\omega_2}{\omega_1}\right)\,,
\end{equation*}
and the entropy
\begin{equation*}
	\mathcal{F}[\rho]:=\int_E \rho \log(\rho)\, \omega_1dz+\frac{1}{q}\int_E\rho|z|^q\omega_2dz\,.
\end{equation*}
We notice that, at least formally, we have $\mathcal{F}[W]=0$, however such an identity could hold only under some integrability conditions for $W$. Our main theorem in the current case is the following.
\begin{theorem}\label{thm:gamma=1} Let $\omega_1, \omega_2$ be homogeneous weights on $E\subset\mathbb{R}^d$ with homogeneity $h_1>-d$ and $h_2$ such that $W\in \mathrm{G}^{q,1}$; assume furthermore that $\beta:=q-h_1+h_2>0 $.  Then, for each function $\rho\in \mathrm{G}^{q,1}$ the following inequality holds
\begin{equation}\label{eqlog}
		\frac{(d+h_1)}{\beta}\log\left(C_W\frac{\beta e }{q}\int_E \frac{\rho}{\norm{\rho}_{L^1_{\omega_1}(E)}} |z|^q \omega_2 dz\right)\norm{\rho}_{L^1_{\omega_1}(E)}
		\geq -\int_E\rho\log\left(\frac{\rho}{\norm{\rho}_{L^1_{\omega_1}(E)}}\right)\, \omega_1 dz
\end{equation}
where $C_W:=\int_E W\omega_1 dz$.  Equality~\eqref{eqlog} is attained if and only if $\rho=W$, up to multiplicative constants and scaling.
\end{theorem}
\begin{proof}
The proof follows the line of the proof of Theorems~\ref{thm1} and~\ref{thm2}. Let $\rho\in \mathrm{G}^{q,1}$ such that $\int_E\rho \omega_1dz=C_W$. Using the convexity of the function $x\mapsto x\log x$ we have
\begin{equation*}
		\rho \log(\rho)-W\log (W)-(\log (W)+1)(\rho-W)\geq 0.
\end{equation*}
Integrating with respect to $\omega_1$ and using the definition of $W$ we find $\mathcal{F}[\rho]\geq \mathcal{F}[W]=0$. Define $\rho_s(z):=s^{d+h_1}\rho(s z)$, then
\begin{equation*}
\begin{split}
		\mathcal{F}[\rho_s]=&s^{d+h_1}\int_E \rho(s z)((d+h_1)\log s+\log(\rho(s z)))\omega_1 dz+\frac{s^{d+h_1}}{q}\int_E \rho(s z)|z|^q\omega_2 dz\\
		=&\int_E\rho\log(\rho)\, \omega_1 dz +(d+h_1)\log s \int_E \rho \omega_1 dz+\frac{s^{-q+h_1-h_2}}{q}\int_E \rho |z|^q\omega_2 dz\\
		=&\int_E\rho\log(\rho)\, \omega_1 dz+A \log s+Bs^{-\beta}=:\int_E\rho\log(\rho)\, \omega_1 dz+g(s)\,,
\end{split}
\end{equation*}
where $A=\left(d+h_1\right)\,C_W$ and $B=\int_E \rho |z|^q\omega_2 dz/q$. The minimum of $g$ is attained at $s_*=\left(B\,\beta/A\right)^\frac{1}{\beta}$ and so
\begin{align*}
		\mathcal{F}[\rho_s]\geq \mathcal{F}[\rho_{s_*}]=g(s^*)= \int_E\rho\log(\rho)\, \omega_1+\frac{A}{\beta}\left(\log\left(\frac{B\beta}{A}\right)+1\right)\geq \mathcal{F}[W]=0.
\end{align*}
We then obtain
\begin{equation*}
\begin{split}
		\int_E\rho\log(\rho)\, \omega_1 dz+\frac{(d+h_1)}{q-h_1+h_2}\left(\int_E\rho \omega_1dz\right)\left(\log\left(\frac{(q-h_1+h_2)e}{q}\int_E \rho |z|^q \omega_2 dz\right)\right)\\
		\geq \frac{(d+h_1)}{q-h_1+h_2}\left(\int_E\rho \omega_1 dz\right)\log\left(\int_E\rho \omega_1 dz\right)
\end{split}
\end{equation*}
and inequality~\eqref{eqlog} follows. Performing the same computations with $W$ in place of $\rho$ we get that $W$ attains the equality.

\noindent Let us now classify all optimal functions for \eqref{eqlog}. Suppose $\overline{W}\in \mathrm{G}^{q,1}$ is such an optimal function for the entropy functional $\mathcal{F}$, see again Remark~\ref{remark-optimisers}. Up to a multiplication for a positive constant, suppose that $\int_E \overline{W} \omega_1 dz=C_W$. Up to scaling (since \eqref{eqlog} is invariant under scaling) suppose also that $\mathcal{F}[\overline{W}]=0$, since $\overline{W}$ is optimal function for \eqref{eqlog} by assumption. Take $\varphi\in C^\infty_c(E)$ and observe that
\begin{align*}
	\frac{d}{d\varepsilon}\mathcal{F}[\psi_\varepsilon]\lvert_{\varepsilon=0}=0, \quad\text{where }\psi_\varepsilon:=\frac{\overline{W}+\varepsilon \varphi}{\int_{E}(\overline{W}+\varepsilon \varphi)\omega_1dz}.
\end{align*}
By a simple computation, we deduce
\begin{align*}
    \frac{d}{d\varepsilon}\mathcal{F}[\psi_\varepsilon]\lvert_{\varepsilon=0}=\int_E\varphi[\omega_1 \log(\overline{W})+\frac{1}{q}|z|^q\omega_2]dz=0
\end{align*}
concluding that $\overline{W}=W$, and the proof is concluded.
\end{proof}

%%%%%%%%%%%%%%%%%%%%%%%%%%%%%%%%%%%%%%%%%%%%%%%%%%%%%%
%%%%%%%%%%%%%%%%%%%%%%%%%%%%%%%%%%%%%%%%%%%%%%%%%%%%%%
\section{Application to weighted Gagliardo-Nirenberg inequalities}\label{secgagnir}
Let us conclude the paper with some applications and examples concerning weighted versions of the Gagliardo-Nirenberg-Sobolev inequalities~\eqref{eqGN}. The main application is a rigidity result for the following weighted Gagliardo-Nirenberg inequality, as obtained in \cite[Theorem 1.1]{Balogh_2025}. Consider
\begin{equation}\label{eqgninequality}
 	\mathcal{C}_{BDK} \left(\int_E |u|^{\sigma p} \omega_1 dz\right)^\frac{1}{\sigma\,p} \leq \left(\int_{E}|\nabla u|^p\,\omega_2dz\right)^\frac{1-\theta}{p}\left(\int_E |u|^{\sigma\,p\gamma}\omega_3dz\right)^\frac{\theta}{\sigma\,p\gamma}
\end{equation} 
 where 
\begin{equation}\label{parameters-GNS}
 	d\ge 2\,,\quad 1< p<\infty\,, \quad 1>\gamma>\max\{1-\frac1d, \frac1{q}\}\,,\quad\sigma=\frac{1}{1+p(\gamma-1)}\,,
\end{equation}
and the value of $\theta$ is deduced by scaling properties. As explained in the introduction, the inequality is posed on an open convex cone $E\subset \mathbb{R}^d$ where the three weights $\omega_1,\,\omega_2,\,\omega_3 \in C^1(E)$ are homogeneous weights of degree denoted, respectively, by $h_1,\,h_2,\,h_3$. 
We also assume that
\begin{equation*}
h_3=\frac{h_1}{q}+\frac{h_2}{p}\,.
\end{equation*}
The main assumption on the weights in this case is the following, see also~\cite[Theorem 1.1]{Balogh_2025}
\begin{equation}\label{eqmain}
\begin{split}
     \left(\frac{1}{1-\gamma}-d\right)\left( \left( \frac{\omega_2(y)}{\omega_2(x)}\right)^\frac{1}{p}\left(\frac{\omega_1(y)}{\omega_1(x)}\right)^{\frac{1}{q}-\gamma}\right)^\frac{1}{1-d(1-\gamma)}\leq& \left(\frac{1}{1-\gamma}+K\right) \frac{\omega_3(x)}{\omega_1^{1/q}(x)\omega_2^{1/p}(x)}\\
     &+C_0\left(\frac{1}{p}\frac{\nabla \omega_2(x)}{\omega_2(x)}+\frac{1}{q}\frac{\nabla \omega_1(x)}{\omega_1(x)} \right)\cdot y\,,
\end{split}
\end{equation}
where $K\in \R$, $C_0>0$ and $x,y\in E$. Remark that \eqref{eqmain} is equivalent to a curvature dimension condition $CD(0,d+h_1)$ if all weights are equal and $h_1$ is the common degree of homogeneity. Here is a refined version of~\cite[Theorem 1.3]{Balogh_2025}. In what follows, we consider that a function attains equality in~\eqref{eqgninequality} if assumption~\eqref{H_1} (as stated in the introduction) holds. 
\begin{corollary}\label{corrig}
Let us~\eqref{parameters-GNS},~\eqref{eqmain} hold and that $\frac{1}{1-\gamma}+K>0$.
Under the assumptions of Theorem~\ref{thm1}, if there exists a non-trivial function attaining equality in \eqref{eqgninequality} with constant $\mathcal{C}_{BDK}$, then the three weights are equal up to multiplicative constants.
 \end{corollary}
 \begin{proof}
     It is enough to apply \cite[Theorem 1.2]{Balogh_2025} together with Theorem \ref{thm1}, which ensures that the quotient $\mathcal{Q}_{\omega_1, \omega_2}$ \eqref{quotient} is attained by a function.
 \end{proof}
 Let us make a concrete example with monomial weights. Let $A,B,C\in \mathbb{R}_+^d$ be three vectors and denote by $h_1:=\sum_{i=1}^d A_i$, $h_2:=\sum_{i=1}^d B_i$, $h_3:=\sum_{i=1}^d C_i$ the sum of their components. Define three monomial weights as 
 \begin{align*}
     \omega_1(z):=z^A:=\prod_{i=1}^d z_i^{A_i},\quad \omega_2(z):=z^B:=\prod_{i=1}^d z_i^{B_i},\quad 
     \omega_3(z):=z^C:=\prod_{i=1}^d z_i^{C_i},
 \end{align*}
 and the cone $E$ as
 \begin{align*}
     E=\{z=(z_1,z_2,...,z_d)\in \mathbb{R}^d:z_i>0\,\text{whenever}\,C_i>0\}.
 \end{align*}
 It was proved in \cite[Example 5.1]{Balogh_2025} via a generalised AM-GM inequality that these weights satisfy \eqref{eqmain} for some constants $K$ and $C_0$ where in particular
 \begin{equation*}
     K=-\frac{1}{1-\gamma}+C_0\left( -d-h_1+\frac{1}{1-\gamma}\left(1+\frac{h_1-h_2}{p}\right)\right).
 \end{equation*}
 Therefore, $\frac{1}{1-\gamma}+K>0$ if and only if 
 \begin{equation}\label{eq6.5}
     \frac{d+\frac{h_1}{q}+\frac{h_2}{p}-1}{d+h_1}<\gamma.
 \end{equation}

 Moreover, \eqref{hp1} is satisfied if and only if 
 \begin{equation}\label{eq6.7}
     \frac{A_i}{q}+\frac{B_i}{p}+1>\gamma(1+A_i).
 \end{equation}
Then, if we sum \eqref{eq6.7} over $i\in \{1,..,d\}$ we obtain \eqref{hp2}. 
So, under hypotheses \eqref{eq6.7}, \eqref{eq6.5} and \eqref{hp3}, Corollary \ref{corrig} tells that inequality \eqref{eqgninequality} cannot admit an optimal function unless $A_i=B_i=C_i$ for all $i$.\\ \ \\
To conclude, we observe that there are some homogeneous weights which do not necessarily satisfy condition \eqref{eqmain}. For instance, it is necessary to add a radial component to monomial weights to find that the condition does not hold any longer: let us show it in the following lines, considering the weights from \cite{pagliarin2026sharp}.
Consider the limiting case of the Sobolev inequality in \eqref{eqgninequality} with $\theta_1=1$ and $\omega_3=0$ (by using the convention that $0^0=1$) and define $\omega_1(z):=z^A|z|^\xi$, $\omega_2(z):=z^A |z|^\eta$. For simplicity assume $p=2$ and $A_d=0$. Then the RHS of \eqref{eqmain} is equal to
\begin{align*}
&\frac{1}{2}\left(\frac{\nabla \omega_1(x)}{\omega_1(x)}+\frac{\nabla \omega_2(x)}{\omega_2(x)}\right)\cdot y=\frac{1}{2}\left(\xi+\eta\right)\frac{x\cdot y}{|x|^2}+\frac{1}{2}\sum_{i=1}^{d-1}A_i\frac{y_i}{x_i}\\=&\frac{1}{2}\left(\xi+\eta\right)\frac{x_d y_d}{|x|^2}+\frac{1}{2}\sum_{i=1}^{d-1}\left(A_i\frac{y_i}{x_i}+\left(\xi+\eta\right)\frac{x_iy_i}{|x|^2}\right)
\end{align*} 
If we choose $x=(1,...,1)$ and $y=(1,1,...,1,-\text{sgn}(\xi+\eta)k)$, with $k>0$ to be determined, we obtain
\begin{align*}
-\frac{1}{2}|\xi+\eta|\frac{k}{d}+\frac{1}{2}h_1+\frac{1}{2}(\xi+\eta)\frac{d-1}{d}
\end{align*}
which is negative if $\xi+\eta\neq 0$ for $k$ sufficiently big, contraddicting \eqref{eqmain}.
 \section{Acknowledgements}
This project has received funding from the European Union's Horizon Europe research and innovation programme under the Marie Sklodowska-Curie grant agreement No 101126554.
\section{Disclaimer}
Co-Funded by the European Union. Views and opinions expressed are however those of the author only and do not necessarily reflect those of the European Union. Neither the European Union nor the granting authority can be held responsible for them.
\addcontentsline{toc}{section}{~~~References}
\bibliographystyle{siam}\small
\bibliography{references}
\end{document}